\documentclass[preprint,authoryear]{elsarticle}
\usepackage[T1]{fontenc}
\usepackage[utf8]{inputenc}
\usepackage{tabularx}
\usepackage{supertabular}
\usepackage{booktabs}
\RequirePackage[authoryear]{natbib}
\usepackage{lmodern}
\usepackage{setspace}
\usepackage{nicefrac}
\usepackage{amsmath,dsfont}
\usepackage{mathtools,mathrsfs}
\usepackage{a4wide}
\usepackage{amssymb}
\usepackage{amsthm}
\usepackage{framed}
\usepackage{bbm}
\usepackage{changes}
\usepackage[flushleft]{threeparttable}
\definecolor{dblue}{rgb}{0.21,0.21,0.55}
\usepackage{multirow}
\usepackage{lscape}
\usepackage[english]{babel}
\usepackage{graphicx}
\usepackage{arydshln}
\usepackage{rotating}
\usepackage[flushmargin,hang]{footmisc}
\usepackage{amsmath}
\usepackage{lscape}
\usepackage{marginnote,letltxmacro}
\usepackage[pdflinkmargin=5pt,pdfstartview={FitBH -32768},plainpages=false]{hyperref}
\usepackage{xcolor}
\usepackage{bbm}

\renewcommand{\P}{\mathbb{P}}

\newcommand{\E}{\mathbb{E}}
\newcommand{\N}{\mathbb{N}}
\newcommand{\R}{\mathbb{R}}
\newcommand{\1}{\mathbbm{1}}
\newcommand{\KLEINO}{{\scriptstyle{\mathcal{O}}}}
\DeclareMathAccent{\verywidehat}{\mathord}{largesymbols}{'144}

\allowdisplaybreaks[3]

\newtheorem{prop}{Proposition}[section]
\newtheorem{lem}{Lemma}
\newtheorem{cor}[prop]{Corollary}

\newtheorem{theo}[prop]{Theorem}

\begin{document}
\renewcommand*{\thefootnote}{\fnsymbol{footnote}}

\title{Cusum tests for changes in the Hurst exponent and volatility of fractional Brownian motion}
%\runtitle{Change point inference on volatility in noisy It\^{o} diffusions}
\author[1]{Markus Bibinger\footnote{Financial support from the Deutsche Forschungsgemeinschaft (DFG) under grant 403176476 is gratefully acknowledged.}}%\footnote[1]{bibinger@mathematik.uni-marburg.de}%\footnote[1]{Corresponding author: bibinger@mathematik.uni-marburg.de, Hans-Meerwein-Str. 6, 35032 Marburg,\\ tel.: 004964212825464, fax: 004964212825469}}
%\author[2]{Mehmet Madensoy\footnote{Madensoy acknowledges financial support from the Deutsche Forschungsgemeinschaft via RTG 1953.}}
\address[1]{Faculty of Mathematics and Computer Science, Philipps-Universit\"at Marburg} 
%\address[2]{School of Business Informatics and Mathematics, Mannheim University}%\\ A5, 6, 68131 Mannheim, Germany,}
\normalsize
\begin{frontmatter}
%\maketitle\thispagestyle{empty}
%\onehalfspacing
%\date{This version: 06.09.2017}

\vspace{-.7cm} 

\begin{abstract}
{{\normalsize \noindent
In this letter, we construct cusum change-point tests for the Hurst exponent and the volatility of a discretely observed fractional Brownian motion. As a statistical application of the functional Breuer-Major theorems by \cite{BEGYN} and \cite{fBM}, we show under infill asymptotics consistency of the tests and weak convergence to the Kolmogorov-Smirnov law under the no-change hypothesis. The test is feasible and pivotal in the sense that it is based on a statistic and critical values which do not require knowledge of any parameter values. Consistent estimation of the break date under the alternative hypothesis is established. We demonstrate the finite-sample properties in simulations and a data example.
%distinguishes paths with constant Hurst exponent and volatility from paths with changes and is \cite{ssnfractals} \cite{taqqufractal}
}}
\begin{keyword}
Change-point test\sep cusum \sep fractional Brownian motion \sep high-frequency data \sep Hurst exponent\sep sunspot\\[.25cm] %power variations\\[.25cm]
%% MSC codes here, in the form: 
\MSC[2010] 60G22\sep 62M10%62M10 \sep 62G10
%C 13, C 58
\end{keyword}
\end{abstract}
% \singlespacing 

%\vspace{.5cm}
\end{frontmatter}
\thispagestyle{plain}

\onehalfspacing
%\doublespacing
\section{Introduction\label{sec:1}}
The (simplest) definition of fractional Brownian motion (fBm), $B^H=(B_t^H)_{t\ge 0}$, is that $B^H$ is a continuous-time Gaussian process with continuous paths which is centered, $\E[B_t^H]=0$ for all $t$, and with covariance function 
\[\E[B_t^H\,B_s^H]=\frac12 (t^{2H}+s^{2H}-|t-s|^{2H})\,,t,s\ge 0\,.\]
$B^H$ has stationary Gaussian increments $(B^H_t-B^H_s)\sim N(0,|t-s|^{2H})$ and is $H$-self-similar. Except the case $H=1/2$, when $B^{1/2}$ is Brownian motion, increments are not independent and $B^H$ is not a semi-martingale. We refer to \cite{nourdin} for more information on properties of fBm. We observe a path of fBm $B_{\sigma}^H=\sigma B^H$, with a volatility (or scaling) parameter $\sigma>0$, over the unit interval $[0,1]$ at $(n+1)$ equidistant discrete observation times, $\sigma B^H_{j \Delta_n},j=0,\ldots,n=\Delta_n^{-1}$. Denote $\Delta_j^n B_{\sigma}^H=\sigma(B^H_{j\Delta_n}-B^H_{(j-1)\Delta_n}),j=1,\ldots,n$, the observed increments called fractional Gaussian noise. We establish limit results under infill asymptotics as $n\to\infty$, $\Delta_n\to 0$.\\
The self-similarity parameter $H\in(0,1)$, called \emph{Hurst exponent}, determines the persistence and the smoothness regularity of paths of $B_{\sigma}^H$. Modelling data by fBm asks for statistical inference on $H$ and $\sigma$. There is a rich literature on estimators for $H$, but local asymptotic normality (LAN) and thus asymptotic efficiency of maximum likelihood (mle) and the whittle estimator under infill asymptotics has been established only recently by \cite{brouste}. Since the computation of the mle is infeasible for too large sample sizes due to the inversion of the covariance matrix, method of moment approaches based on power variations are prominent alternative methods which are simple to implement and fast to compute. We first restrict to the more standard case $H\in(0,3/4)$, \footnote{For $H>3/4$, non-central limit theorems with slower rates are given in \cite{tudor}. We shall use a functional central limit theorem for a sum of squared second-order increments from \cite{BEGYN} for the test in this case.} in that the following central limit theorem, when $\sigma=1$, for the normalized quadratic variation is available:
\begin{align}\label{clt}\sqrt{n}\Big(n^{2H-1}\sum_{j=1}^{ n }(\Delta_j^n B_1^H)^2-1\Big)\stackrel{d}{\longrightarrow}N(0,2\gamma)\,,\end{align}
with $\gamma$ given in \eqref{gamma}, see, for instance, \cite{tudor}. This setup is covered by the well-known theorem of \cite{breuer}. From this convergence, we derive one standard estimator for the Hurst exponent
\begin{align}\label{Hest1}\hat H_n=\frac12-\frac{\log\Big(\sum_{j=1}^{n }(\Delta_j^n B_1^H)^2\Big)}{2\log(n)}~~~(\text{only valid for~}\sigma=1)\,,\end{align}
and based on the delta method that $\sqrt{n}\log(n)(\hat H_n-H)\stackrel{d}{\longrightarrow}N(0,\gamma/2)$, see (6.21) in \cite{nourdin}. In general, if $\sigma=1$ is not true, however, $\hat H_n$ is not a legitimate estimator and does not converge to $H$. For general unknown $\sigma$, a modification of this statistic using a ratio of discrete quadratic variations at different sampling frequencies 
\begin{align}\label{Hest2}\hat H_n^{(\sigma)}=\frac{1}{2\log(2)}\log\bigg(\frac{\sum_{j=0}^{n-2}(\sigma(B^H_{(j+2)\Delta_n}-B^H_{j\Delta_n}))^2}{\sum_{j=1}^{n }(\Delta_j^n B_{\sigma}^H)^2}\bigg)\,,\end{align}
yields a consistent estimator with $\sqrt{n}$-rate. More refined related estimation methods have been presented by \cite{Coeurjolly2} and by \cite{IR} using increment ratios. Lower bounds for the convergence rates of unbiased estimators of $H$ are known to be $\sqrt{n}\log(n)$ when $\sigma$ is known and $\sqrt{n}$ in case of unknown volatility, see \cite{Coeurjolly3} or \cite{brouste}. Hence, the estimators \eqref{Hest1} and \eqref{Hest2} using power variations attain these optimal rates. The volatility can be estimated when we plug-in the estimated Hurst exponent. If $H>3/4$, these estimators preserve consistency, but the central limit theorems do no longer apply.\\
While there is a strand of literature that addresses change-point analysis for changes in the mean of fBm or related time-series models, see, for instance, \cite{betken2017} for a recent contribution, there is scant groundwork on change points of the Hurst exponent $H$. There are a few works in the time-series literature, see \cite{lavancier} and references therein, which do not include a classical cusum test based on power variations. Although cusum tests are in general very popular because of their appealing asymptotic and finite-sample properties, to the best of the author's knowledge the presented methods have not yet been discussed in the literature.\\
A necessary prerequisite for a cusum change-point test is a functional central limit theorem, also referred to as invariance principle. Proving such a result for power variations of fBm is difficult due to the dependence of increments. Nevertheless, exploiting properties of Gaussian processes such results have been established in the last years. Building upon two functional central limit theorems by \cite{BEGYN} and \cite{fBM}, we construct cusum change-point tests for a change in the Hurst exponent or the volatility parameter of the fBm. Note that we use the result from \cite{fBM} which includes general functions that admit a Hermite expansion of the Gaussian random variables only for one simple specific function. %, while the authors focus on necessary and sufficient conditions under which functional weak convergence holds true. 
\section{Statistical application of functional central limit theorems for power variations\label{sec:2}}
The result by \cite{fBM} applied with the Hermite polynomial $H_2(x)=x^2-1$ of rank 2 provides for $H\in(0,3/4)$ a functional limit theorem for the normalized discrete quadratic variation
\begin{align}\label{fclt}\sqrt{n}\Big(n^{2H-1}\sum_{j=1}^{\lfloor n t\rfloor}(\Delta_j^n B_{\sigma}^H)^2-t\sigma^2\Big)\longrightarrow \sqrt{2\gamma}\sigma^2 \, W_t\,,\end{align}
as $n\rightarrow \infty$, weakly in the Skorokhod space of c\'{a}dl\'{a}g functions with a standard Brownian motion $W$ independent of $B_{\sigma}^H$. The limit process being continuous, convergence holds with respect to the uniform topology. The long-run variance is determined by
\begin{align}\label{gamma}
\gamma=\frac14 \sum_{r\in \mathds{Z}}\big(|r+1|^{2H}+|r-1|^{2H}-2|r|^{2H}\big)^2=\frac12\Big(1+\sum_{r=2}^{\infty} (D_2(({\cdot})^{2H},r))^2\Big)\,,
\end{align}
where for some function $f$, the second-order increment operator is $D_{2}(f,r):=f(r)-2f(r-1)+f(r-2), r\in\mathds{Z}$. The estimate $\big|D_2(({\cdot})^{2H},r)\big|\le |r-1|^{-|2H-1|}2H|1-2H|$, $r\ge 2$, %,r$\big||r+1|^{2H}+|r-1|^{2H}-2|r|^{2H}\big|\le |r-1|^{-|2H-1|}2H|1-2H|$, $|r|\ge 2$, 
$H\in(0,3/4)$, shows the convergence of the series, while the sum of squared increments exhibits long-range dependence for $H\ge 3/4$. By continuous mapping, \eqref{fclt} readily implies that
%\begin{subequations}
\begin{align}\label{cusum}\frac{1}{\sqrt{n}}\sum_{j=1}^{\lfloor n t\rfloor}\Big((\Delta_j^n B_{\sigma}^H)^2n^{2H}-\sum_{j=1}^n(\Delta_j^n B_{\sigma}^H)^2n^{2H-1}\Big)\longrightarrow \sqrt{2\gamma}\sigma^2\big(W_t-tW_1\big)\,,
\end{align}
%\end{subequations}
that is, the convergence of the cusum process to a Brownian bridge. This functional central limit theorem can be exploited for a change-point test with the help of the next simple lemma.
\begin{lem}\label{cmt}Let $X$ and $Y$ be real-valued c\'{a}dl\'{a}g functions defined on $[0,1]$. For any $\epsilon>0$, there exists $\delta>0$, such that 
\[\|X-Y\|_{\infty}=\sup_{0\le t \le 1}|X_t-Y_t|\le\delta~~\Rightarrow ~~ \Big|\sup_{0\le t \le 1} X_t-\sup_{0\le t \le 1} Y_t\Big|\le \epsilon\,.\]
\end{lem}
\begin{proof}If $\big|\sup_{0\le t \le 1} X_t-\sup_{0\le t \le 1} Y_t\big|>\epsilon$, there exists $s\in[0,1]$, such that
\begin{align*}X_s-Y_s\ge X_s-\sup_{0\le t \le 1} Y_t>\epsilon~~~\mbox{or}~~~~Y_s-X_s&\ge Y_s-\sup_{0\le t \le 1} X_t>\epsilon\,,\end{align*}
and thus $\|X-Y\|_{\infty}>\epsilon$. Hence, for $\epsilon>0$ and $\delta=\epsilon$, we obtain the claimed continuity.
\end{proof}
Based on Lemma \ref{cmt} and continuous mapping, we obtain the weak convergence of the supremum of the absolute left-hand side in \eqref{cusum}, multiplied with $(\sqrt{2\gamma}\sigma^2)^{-1}$, to the law of the supremum of the absolute value of a standard Brownian bridge, referred to as the Kolmogorov-Smirnov law. %This result extends to a statistic where the normalization uses consistent estimates of $\gamma$ and $\sigma^2$. 
In the vein of \cite{phillips}, this weak convergence can be exploited to test for structural breaks in the observed path of $B_{\sigma}^H$. Depending, however, on $H$, $\sigma^2$ and $\gamma$, this statistic is yet infeasible and we thus propose the following modification
\begin{align}\label{scusum}T_n&=\sup_{1\le m\le n}\big|S_{n,m}\big|\,,~\mbox{with}~\\
\notag S_{n,m}&=\frac{\tfrac{1}{\sqrt{n}}\sum_{j=1}^{m}\Big((\Delta_j^n B_{\sigma}^H)^2-\tfrac{1}{n}\sum_{j=1}^n(\Delta_j^n B_{\sigma}^H)^2\Big)}{\bigg(\hspace*{-.05cm}\sum\limits_{k=-n+1}^{n-1}\frac{1}{n}\sum\limits_{j=1\wedge (1-k)}^{(n-k)\wedge n}\Big((\Delta_j^n B_{\sigma}^H)^2-\tfrac{\sum_{j=1}^n(\Delta_j^n B_{\sigma}^H)^2}{n}\Big)\Big((\Delta_{j+k}^n B_{\sigma}^H)^2-\tfrac{\sum_{j=1}^n(\Delta_j^n B_{\sigma}^H)^2}{n}\Big)\hspace*{-.05cm}\bigg)^{\frac12}}\,,
%T_n\hspace*{-.05cm}=\hspace*{-.05cm}\sup_{1\le m\le n}\frac{\tfrac{1}{\sqrt{n}}\sum_{j=1}^{m}\Big((\Delta_j^n B_{\sigma}^H)^2-\tfrac{1}{n}\sum_{j=1}^n(\Delta_j^n B_{\sigma}^H)^2\Big)}{\Big(\hspace*{-.05cm}\sum_{k\in\mathds{Z}}\tfrac{1}{n}\sum_{j=1}^{n-k}\Big((\Delta_j^n B_{\sigma}^H)^2-\tfrac{\sum_{j=1}^n(\Delta_j^n B_{\sigma}^H)^2}{n}\Big)\Big((\Delta_{j+k}^n B_{\sigma}^H)^2-\tfrac{\sum_{j=1}^n(\Delta_j^n B_{\sigma}^H)^2}{n}\Big)\hspace*{-.05cm}\Big)^{\frac12}}\,.
\end{align}
where we write $a\wedge b=\min(a,b)$. $S_{n,m}$ is a cusum process of squared increments standardized with the square root of the empirical long-run variance, that is, a sum over empirical autocovariances for different lags. It holds that 
\begin{align}\label{scusumcon}T_n\stackrel{d}{\longrightarrow} \sup_{0\le t\le 1}\big|W_t-tW_1\big|\,,~~~\mbox{as}~n\to\infty\,,\end{align} 
with the Kolmogorov-Smirnov limit law. The standardization with the square root of the empirical long-run variance thus takes out the factor $\sqrt{2\gamma}\sigma^2$ in \eqref{fclt}. Moreover, we do not rescale squared increments by $n^{2H}$ in \eqref{scusum}. Multiplying numerator and denominator with $n^{2H}$, the squared denominator consistently estimates the long-run variance as $n\to\infty$, and it is positive definite by standard results from time series analysis. Slutsky's lemma in combination with the weak convergence for the infeasible statistic based on \eqref{cusum} thus gives the stated asymptotic distribution of $T_n$. %This result extends to a statistic where the normalization uses consistent estimates of $\gamma$ and $\sigma^2$. 
In the next paragraph, we work out statistical properties of $T_n$ and construct a change-point test. Let us recall that \eqref{scusumcon} holds true only in the case $H<3/4$. Instead of considering functional non-central limit theorems when $H\ge 3/4$, if they were available at all, and the law of the supremum of a Rosenblatt limit process, a simpler solution to address the general case, $H\in(0,1)$, is to use an analogous statistic as \eqref{scusum} with second-order increments. For observations $(Z_j)_{0\le j\le n}$, set
\begin{align}\label{sscusum}T_n^{(2)}&=\sup_{1\le m\le n-1}\big| S_{n,m}^{(2)}\big|\,,~\mbox{with}~\\
\notag S_{n,m}^{(2)}&=\frac{\tfrac{1}{\sqrt{n}}\sum_{j=1}^{m}\Big((D_2(Z_j))^2-\tfrac{1}{n-1}\sum_{j=1}^{n-1}(D_2(Z_j))^2\Big)}{\bigg(\hspace*{-.1cm}\sum\limits_{k=-n+2}^{n-2}\frac{1}{n-1}\hspace*{-.1cm}\sum\limits_{j=1\wedge (1-k)}^{(n-k-1)\wedge (n-1)}\hspace*{-.1cm}\Big(\hspace*{-.05cm}(D_2(Z_j))^2\hspace*{-.05cm}-\hspace*{-.05cm}\tfrac{\sum_{j=1}^{n-1}(D_2(Z_j))^2}{n-1}\hspace*{-.05cm}\Big)\hspace*{-.05cm}\Big(\hspace*{-.05cm}(D_2(Z_{j+k}))^2\hspace*{-.05cm}-\hspace*{-.05cm}\tfrac{\sum_{j=1}^{n-1}(D_2(Z_j))^2}{n-1}\hspace*{-.05cm}\Big)\hspace*{-.1cm}\bigg)^{\frac12}},
%T_n\hspace*{-.05cm}=\hspace*{-.05cm}\sup_{1\le m\le n}\frac{\tfrac{1}{\sqrt{n}}\sum_{j=1}^{m}\Big((\Delta_j^n B_{\sigma}^H)^2-\tfrac{1}{n}\sum_{j=1}^n(\Delta_j^n B_{\sigma}^H)^2\Big)}{\Big(\hspace*{-.05cm}\sum_{k\in\mathds{Z}}\tfrac{1}{n}\sum_{j=1}^{n-k}\Big((\Delta_j^n B_{\sigma}^H)^2-\tfrac{\sum_{j=1}^n(\Delta_j^n B_{\sigma}^H)^2}{n}\Big)\Big((\Delta_{j+k}^n B_{\sigma}^H)^2-\tfrac{\sum_{j=1}^n(\Delta_j^n B_{\sigma}^H)^2}{n}\Big)\hspace*{-.05cm}\Big)^{\frac12}}\,.
\end{align}
where $D_2(Z_j)=D_2(\text{id},Z_j)=Z_{(j+1)\Delta_n}-2Z_{j\Delta_n}+Z_{(j-1)\Delta_n}$ are second-order increments. A normalized sum of squared second-order increments of fBm satisfies a (functional) central limit theorem for any $H\in(0,1)$, see Section 3.1 of \cite{BEGYN} and we thus readily derive weak convergence of $T_n^{(2)}$ to the Kolmogorov-Smirnov limit law for any $H\in(0,1)$.
\section{A cusum test for changes in the Hurst exponent\label{sec:3}}
We consider the statistical hypothesis test
\begin{equation*}
\begin{aligned}
&H_{0}:\text{We observe a discretization of a path of $B_{\sigma}^H$ with $\sigma>0$, $H\in(0,3/4)$ constant }\quad\text{vs.}\\
&H_{1}:\text{There is one }\theta\in\left(0,1\right)\text{ and }\text{we have $B_{\sigma}^H$ on $[0,\theta)$ and $B_{\sigma}^{\tilde H},|H-\tilde H|>0$, on $[\theta,1]$}.
\end{aligned}
\end{equation*}

It is standard in the theory of statistics for high-frequency data to address such questions \emph{path-wise}. This means that $H_0$ and $H_1$ are formulated for the one observed path of $(\sigma\,B_{t}^H)_{t\in[0,1]}$, and the statistical decision is based on the discretization of the given path. It is not important here how the inter-dependence structure evolves under $H_1$ around $\theta$. One can think for instance of two independent fBms intertwined at the change, but a gradual change of persistence in a vicinity of $\theta$ is possible as well. %We impose a regularity exponent $\aalpha \in (0,1]$, being the Hölder exponent and\mn{'being the Hölder exponent and' ergänzt} governing the smoothness of $(\sigma_t)_{t\ge 0}$ under $H_0$. 
Next, we give the main result for the proposed cusum test.
\begin{theo}\label{thm}The sequence of tests with critical regions
\begin{align}\label{critc}C_n=\big\{T_n>q_{1-\alpha}\}\,,\end{align}
%\footnote{That is, the asymptotic distribution does not hinge on any unknown parameters.}  
where $q_{1-\alpha}$ denotes the $\alpha$-fractile, that is, the (1-$\alpha$)-quantile, of the Kolmogorov-Smirnov law, provides an asymptotic distribution-free test for the null hypothesis $H_0$ against the alternative hypothesis $H_1$, which has asymptotic level $\alpha$ and asymptotic power 1. The consistency is valid for null sequences of $|H-\tilde H|_n$ in $n$, with $|H-\tilde H|_n\log(n)\to 0$, as long as $n^{-1/2}/\log(n)=\KLEINO(|H-\tilde H|_n)$.
\end{theo}
\begin{proof}From \eqref{scusumcon}, we conclude that under $H_0$
\[\P_{H_0}(C_n)\to\alpha~,~\P_{H_0}(T_n\le q_{1-\alpha})\to 1-\alpha\,,\]
such that the test has asymptotic level $\alpha$. Consistency means that the type II error probabilities tend to zero:
\[\P_{H_1}(T_n\le q_{1-\alpha})\to 0\,,~\text{as}~n\to\infty\,.\]
This is implied by the following lower bound for $T_n$ under $H_1$. For the numerator of $|S_{n,\lfloor n\theta\rfloor}|$, we use that
\begin{align} &\notag\tfrac{1}{\sqrt{n}}\Big|\sum_{j=1}^{\lfloor n\theta\rfloor}\Big((\Delta_j^n B_{\sigma}^H)^2-\tfrac{1}{n}\sum_{j=1}^n(\Delta_j^n B_{\sigma}^H)^2\Big)\Big|\\
&\label{numerator}\ge \tfrac{\sigma^2}{\sqrt{n}}\Big|\sum_{j=1}^{\lfloor n\theta\rfloor}\Big(n^{-2H}-\frac{\lfloor n\theta\rfloor}{n} n^{-2H}-\frac{n-\lceil n\theta\rceil+1}{n}n^{-2\tilde H}\Big)\Big|\\
&\notag -\tfrac{1}{\sqrt{n}}\Big|\sum_{j=1}^{\lfloor n\theta\rfloor}\Big((\Delta_j^n B_{\sigma}^H)^2-\tfrac{\sigma^2}{n^{2H}}-\tfrac{1}{n}\sum_{j=1}^{\lfloor n\theta\rfloor}\big((\Delta_j^n B_{\sigma}^H)^2-\tfrac{\sigma^2}{n^{2H}}\big)-\tfrac{1}{n}\sum_{j=\lceil n\theta\rceil}^{n}\big((\Delta_j^n B_{\sigma}^H)^2-\sigma^2n^{-2\tilde H}\big)\Big)\Big|\,,
\end{align}
where we added and subtracted expectations of the squared increments and applied the reverse triangle inequality. For sufficiently large $n$, the obtained lower bound is almost surely positive such that we can drop the absolute value on the right-hand side.\\
In the sequel, we write $H\wedge \tilde H=\min(H,\tilde H)$ and $H\vee \tilde H=\max(H,\tilde H)$. The denominator of $|S_{n,\lfloor n\theta\rfloor}|$, multiplied with $n^{2(H\wedge\tilde H)}$, converges in probability to $\sqrt{2\gamma}\sigma^2\theta$ if $H<\tilde H$, or $\sqrt{2\gamma}\sigma^2(1-\theta)$ if $H>\tilde H$, respectively. This is a direct generalization of the convergence in probability of the denominator, multiplied with $n^{2H}$, under $H_0$ to the long-run standard deviation $\sqrt{2\gamma}\sigma^2$. Therefore, for sufficiently large $n$, an almost sure upper bound of this term is given by $C_{\theta}\sigma^2$ with some constant $C_{\theta}$, for instance, $C_{\theta}=2\sqrt{2\gamma}\sigma^2\max(\theta, (1-\theta))$.
By these upper and lower bounds we obtain that, almost surely for sufficiently large $n$, it holds that
\begin{align}\label{lowbound}T_n\ge |S_{n,\lfloor n\theta\rfloor}|&\ge \frac{\big|\tfrac{1}{\sqrt{n}}\sum_{j=1}^{\lfloor n\theta\rfloor}\Big(n^{-2H}-\frac{\lfloor n\theta\rfloor}{n} n^{-2H}-\frac{n-\lceil n\theta\rceil+1}{n}n^{-2\tilde H}\Big)\big|}{C_{\theta}n^{-2(H\wedge\tilde H)}}-|Z_n|\\[.1cm]
&\notag \ge C_{\theta}^{-1}\,\theta(1-\theta)\big(1+\mathcal{O}(n^{-1/2})\big)\sqrt{n}\big(1-n^{-2(H\vee\tilde H - H\wedge \tilde H)}\big)-\mathcal{O}_{\P}(1)\,.
\end{align}
The sequence of random variables $(Z_n)$ contains the second, compensated addend from the lower bound in \eqref{numerator}. They have expectation zero and converge in distribution by the same central limit theorem as under the null hypothesis and the sequence of their distributions is thus tight, i.e.\ $|Z_n|=\mathcal{O}_{\P}(1)$.\,\footnote{To show only the tightness, it suffices to apply Markov's inequality.} Hence, for fix $|H-\tilde H|$ and $\theta(1-\theta)$, $T_n\to\infty$ at rate $\sqrt{n}$. This proves the consistency of the test.\\ 
%A standard decomposition in the expectation and a bound for the stochastic deviation by Markov's inequality yield with an almost surely finite constant $C$ and an almost surely positive constant $c$ that 
%\begin{align}\label{lowbound}T_n\ge |S_{n,\lfloor n\theta\rfloor}|&\ge \frac{\big|\tfrac{1}{\sqrt{n}}\sum_{j=1}^{\lfloor n\theta\rfloor}\big(n^{-2H}-\theta n^{-2H}-(1-\theta)n^{-2\tilde H}\big)\big|}{Cn^{-2(H\wedge\tilde H)}}-\mathcal{O}_{\P}(1)\\[.1cm]
%&\notag \ge c\,\theta(1-\theta)\sqrt{n}\big(1-n^{-2(H\vee\tilde H - H\wedge \tilde H)}\big)-\mathcal{O}_{\P}(1)\,,
%\end{align}
%where we write $H\wedge \tilde H=\min(H,\tilde H)$ and $H\vee \tilde H=\max(H,\tilde H)$. Hence, when $|H-\tilde H|$ and $\theta(1-\theta)$ are fix, $T_n\to\infty$ with rate $\sqrt{n}$. Naturally, the power of the test decreases for $\theta$ closer to one boundary of the observation interval or for smaller $|H-\tilde H|$.
For a deeper asymptotic analysis of the test, we consider a deterministic null sequence of parameter changes $|H-\tilde H|_n=\delta(n)$. This is often referred to as `local alternative', see, for instance, Section 1 of \cite{ingster}. We determine at what rate the sequence $\delta(n)$ may go to zero, such that our test preserves consistency, that means we determine its `detection boundary'. %\,\footnote{A general introduction to this kind of asymptotic analysis of parameter tests is provided, for instance, in \cite{ingster}.} 
For arbitrary $\epsilon>0$, there are a constant $K_{\epsilon}$ and $n_0\in\N$, such that
\begin{align}\label{op}\P\big(|Z_n|\ge K_{\epsilon}\big)< \epsilon~,~\forall~n\ge n_0\,,\end{align}
by the definition of the property $|Z_n|=\mathcal{O}_{\P}(1)$. By \eqref{lowbound}, we obtain for sufficiently large $n$ that
\begin{align*}\P_{H_1}(T_n\le q_{1-\alpha})&\le \P_{H_1}\big(C_{\theta}^{-1}\,\theta(1-\theta)\big(1+\mathcal{O}(n^{-1/2})\big)\sqrt{n}\big(1-n^{-2\delta(n)}\big)-|Z_n|\le q_{1-\alpha}\big)\,.%\\
%&\le \P_{H_1}\big(C^{-1}\,\theta(1-\theta)\big(1+\mathcal{O}(n^{-1/2})\big)\sqrt{n}\big(1-n^{-2\delta(n)}\big)\le q_{1-\alpha}+K_{\epsilon}\big)+\frac{\epsilon}{2}\,.%~,~\forall~n\ge n_0\,.
\end{align*}
If $\sqrt{n}(1-n^{-2\delta(n)})\to \infty$, we conclude that for sufficiently large $n$, we have that
\[C_{\theta}^{-1}\,\theta(1-\theta)\big(1+\mathcal{O}(n^{-1/2})\big)\sqrt{n}\big(1-n^{-2\delta(n)}\big)>K_{\epsilon}+q_{1-\alpha}\,,\]
such that \eqref{op} implies that there exists $N\in\N$, such that
\begin{align*}\P_{H_1}(T_n\le q_{1-\alpha})< \epsilon ~,~\forall~n\ge N\,.\end{align*}
We finish the proof by showing that the condition $\sqrt{n}(1-n^{-2\delta(n)})\to \infty$ is equivalent to the rate of the detection boundary given in the theorem. For $\delta(n)\log(n)\to 0$, we use that
\begin{align*}
1-n^{-2\delta(n)}=1-\exp\big(-2\delta(n)\log(n)\big)=2\delta(n)\log(n)+\mathcal{O}\big(\delta^2(n)\log^2(n)\big)\,.
\end{align*}
Hence, $\sqrt{n}\delta(n)\log(n)\to\infty$ ensures consistency of the test which is equivalent to $n^{-1/2}/\log(n)=\KLEINO(\delta(n))$.
\end{proof}
%\begin{rem}\rm
The lower bound \eqref{lowbound} shows that the power of the test decreases for $\theta$ closer to one boundary of the observation interval and for smaller $|H-\tilde H|$. This is a natural standard property of a change-point test. Consistency of the test directly extends to the case of more than one change in $H$. Further, a locally bounded drift term added to $B_{\sigma}^H$ will not affect the results, which follows with standard estimates for high-frequency data provided by \cite{JP}. In case of additional nuisance jumps, truncation methods can be used to obtain a robust version.\\[.2cm]
%\end{rem}
Using \eqref{sscusum}, we can drop the restriction $H<3/4$ under $H_0$.
\begin{cor}\label{corscusum}The sequence of tests with critical regions $C_n=\big\{T_n^{(2)}>q_{1-\alpha}\}$, provides an asymptotic distribution-free test for $H_0$, for any $H\in(0,1)$, against $H_1$ with asymptotic level $\alpha$ and asymptotic power 1. The consistency is valid for null sequences of $|H-\tilde H|_n$ in $n$, with $|H-\tilde H|_n\log(n)\to 0$, as long as $n^{-1/2}/\log(n)=\KLEINO(|H-\tilde H|_n)$.
\end{cor}
For the lower bound of $T_n^{(2)}$ under $H_1$, we use that $(D_2(Z_j))_{1\le j\le (n-1)}$ are centered and 
\begin{align}\label{sod}\E[(D_2(Z_j))^2]=\sigma^2(4-2^{2H})\Delta_n^{2H}=c_H\sigma^2\Delta_n^{2H}\,,\end{align}
with $0<c_H<3$, and $c_H$ decreasing in $H$. We determine the lower bound under $H_1$ and the detection boundary analogously as for $T_n$ in the proof of Theorem \ref{thm}. If without loss of generality $H<\tilde H$, an almost sure lower bound for sufficiently large $n$ is given by
\begin{align*}T_n^{(2)}\ge |S_{n,\lfloor n\theta\rfloor}^{(2)}|&\ge c\,\theta(1-\theta)\sqrt{n}\big(c_H-c_{\tilde H}\Delta_n^{2(\tilde H - H)}\big)-\mathcal{O}_{\P}(1)\\
&\ge c\,(4-2^{2H})\theta(1-\theta)\sqrt{n}\big(1-\Delta_n^{2(\tilde H - H)}\big)-\mathcal{O}_{\P}(1)\,,
\end{align*}
with an almost surely positive constant $c$. This proves Corollary \ref{corscusum}.
\section{A cusum test for changes in the volatility\label{sec:4}}
We prove that $T_n$ reacts as well to a change of $\sigma^2$ and thus provides a statistical hypothesis test for
\begin{equation*}
\begin{aligned}
&H_{0}:\text{We observe a discretization of a path of $B_{\sigma}^H$ with $\sigma>0$, $H\in(0,3/4)$ constant }\quad\text{vs.}\\
&H^{\sigma}_{1}:\text{There is one }\theta\in\left(0,1\right)\text{ and }\text{we have $B_{\sigma}^H$ on $[0,\theta)$ and $B_{\tilde\sigma}^{H},|\sigma^2-\tilde \sigma^2|>0$, on $[\theta,1]$}.
\end{aligned}
\end{equation*}
A related lower bound for $T_n$ under $H^{\sigma}_1$ as in \eqref{lowbound} yields with an almost surely finite constant $C$ and an almost surely positive constant $c$ that 
\begin{align*}T_n\ge |S_{n,\lfloor n\theta\rfloor}|&\ge \frac{\big|\tfrac{1}{\sqrt{n}}\sum_{j=1}^{\lfloor n\theta\rfloor}\big(\sigma^2n^{-2H}-\theta \sigma^2n^{-2H}-(1-\theta)\tilde\sigma^2n^{-2 H}\big)\big|}{Cn^{-2H}\max(\sigma^2,\tilde\sigma^2)}-\mathcal{O}_{\P}(1)\\[.1cm]
&\notag \ge c\,\theta(1-\theta)\sqrt{n}\bigg(1-\frac{\min(\sigma^2,\tilde\sigma^2)}{\max(\sigma^2,\tilde\sigma^2)}\bigg)-\mathcal{O}_{\P}(1)\,.
\end{align*}
We obtain the following corollary to Theorem \ref{thm}.
\begin{cor}\label{corr}The sequence of tests with critical regions \eqref{critc}, provides an asymptotic distribution-free test for the null hypothesis $H_0$ against the alternative hypothesis $H_1^{\sigma}$, which has asymptotic level $\alpha$ and asymptotic power 1. The consistency is valid for decreasing null sequences of $|\sigma^2-\tilde \sigma^2|_n$ in $n$, as long as $n^{-1/2}=\KLEINO(|\sigma^2-\tilde \sigma^2|_n)$.
\end{cor}
Thus, the test rejects the null hypothesis under both types of changes. There are opportunities to discriminate the two types of changes based on the different behavior of the ratio of expected squared increments before and after the break at time $\theta$. We present one in the next paragraph. With \eqref{sod} we obtain an analogous result for the test based on \eqref{sscusum} for any $H\in(0,1)$ under $H_0$ as well. In view of the rates in the LAN result by \cite{brouste}, we can see that our tests attain asymptotic minimax-optimal rates. This means that parameter differences smaller than the orders stated in Theorem \ref{thm} and Corollary \ref{corr} are impossible to detect and determine the minimax detection boundary.
\section{Estimation of the change point and discriminating the type of change\label{sec:5}}
If a change-point test rejects the null hypothesis of no change, the estimation of the time of a change, referred to as the break date, becomes of interest. We prove consistency of the argmax-estimator associated with our statistic \eqref{scusum} under the alternative hypothesis with one break in $H$ at time $\theta\in(0,1)$. The result analogously extends to a break in $\sigma^2$, and using an iterative algorithm it may be extended to multiple changes.
\begin{prop}\label{propest}Under the alternative hypothesis $H_1$, the estimator
\begin{align}\label{est}\hat\theta_n=\Delta_n\,\operatorname{argmax}_{1\le m\le n}\big|S_{n,m}\big|\end{align}
satisfies that $|\hat\theta_n-\theta|=\mathcal{O}_{\P}\big(\log(n)n^{-1/2}\big)$.
\end{prop}
\begin{proof}
Without loss of generality, we consider observations $(Z_j)_{0\le j\le n}$, with
\begin{align*}\E[(\Delta_j^n Z)^2]=\begin{cases}\sigma^2n^{-2H}&,\,1\le j\le i_n^*\\ \sigma^2n^{-2\tilde H}&,\,i_n^*+1\le j\le n\end{cases}\,,\end{align*}
where $i_n^*=\theta n\in\N$ and $n^{-2H}=n^{-2\tilde H}+\delta_H$, $\delta_H>0$. Generalizations of the proof to $\delta_H<0$, and when one increment is affected by $H$ and $\tilde H$ are obvious. Define $f:[0,1]\to\R$, a piecewise constant function with
\[f(t)=\sum_{m=1}^n \E\Big[\sum_{j=1}^{m}\Big((\Delta_j^n Z)^2-\frac{1}{n}\sum_{j=1}^n(\Delta_j^n Z)^2\Big)\Big]\,\1_{((m-1)\Delta_n,m\Delta_n]}(t)\,,\]
and $f(0)=0$. We obtain for $0\le m\le n$, that
\begin{align*}f(m\Delta_n)&=\begin{cases}\sigma^2n^{-2H}m\big(1-\frac{i_n^*}{n}\big)-\sigma^2n^{-2\tilde H}\,m\big(1-\frac{i_n^*}{n}\big)&,\,0\le m\le i_n^*\\ \sigma^2n^{-2H}i_n^*\big(1-\frac{m}{n}\big)-\sigma^2n^{-2\tilde H}\,i_n^*\big(1-\frac{m}{n}\big)&,\,i_n^*<m\le n\end{cases}\\
&=\begin{cases}\sigma^2\delta_H m\big(1-\frac{i_n^*}{n}\big)&,\,0\le m\le i_n^*\\ \sigma^2\delta_H i_n^*\big(1-\frac{m}{n}\big)&,\,i_n^*<m\le n\end{cases}\,.
\end{align*}
We see that $f$ is non-negative and increasing for $t< \theta$, and decreasing for $t>\theta$, with a unique maximum at $\theta$. $f(m\Delta_n)/\sqrt{n}$ is the expectation of the numerator of the cusum process in \eqref{scusum}. Thus, $S_{n,m}=f(m\Delta_n)/\sqrt{n}+\mathcal{O}_{\P}(1)$ uniformly in $m$. For $\delta_H>0$, it holds that $T_n>0$, with a probability tending to 1, and the leading term is larger than the $\mathcal{O}_{\P}(1)$-term with a probability tending to 1. Hence, the convergence rate is implied by the relation 
\[\frac{1}{\sqrt{n}}\big(f(i_n^*\Delta_n)-f(i_n^*\Delta_n-\gamma_n)\big)=\sigma^2\delta_H\,\sqrt{n}\,\gamma_n\big(1-\frac{i_n^*}{n}\big)\to\infty\,,\]
as long as $\gamma_n\sqrt{n}\to\infty$. More precisely, we apply Lemma 1 of \cite{mehmet} to the function $f{\big|}_{\left[0,\theta\right]}$, which yields for $\gamma_n\sqrt{n}\to\infty$ that with a probability converging to 1 we have that
\begin{align*}
i_n^* \Delta_n\ge \hat\theta_n \ge i_n^* \Delta_n-\gamma_{n}\,.
\end{align*}
An analogous application of the same lemma to the function $f{\big|}_{\left[\theta,1\right]}$, yields that with a probability converging to 1 it holds that
\begin{align*}
i_n^* \Delta_n \le  \hat\theta_n \le i_n^*\Delta_n+\gamma_{n}\,.
\end{align*}
\end{proof}
One opportunity to distinguish which type of change occurs when the test rejects the null hypothesis uses the estimator $\hat\theta_n$.
\begin{cor}With $\hat\theta_n$ from \eqref{est}, for $\hat\theta_n n \in\{1,\ldots,n-1\}$ and observations $(Z_j)_{0\le j\le n}$, set
\begin{align}Q_n=\frac{\frac{1}{\hat\theta_n n}\sum_{j=1}^{\hat\theta_n n}(\Delta_j^n Z)^2}{\frac{1}{(n-\hat\theta_n n)}\sum_{j=\hat\theta_n n+1}^{n}(\Delta_j^n Z)^2}\,,\end{align}
and $Q_n=0$, else. It holds that 
\begin{align}Q_n\stackrel{\P}{\longrightarrow}\begin{cases}\infty \,,&~\mbox{under}~H_1~\mbox{with}~H<\tilde H\,,\\ \frac{\sigma^2}{\tilde\sigma^2} \,,&~\mbox{under}~H_1^{\sigma}\,,\\ 0 \,,&~\mbox{under}~H_1~\mbox{with}~H>\tilde H\,,\end{cases}\,\end{align}
as $n\to\infty$.
\end{cor}
\begin{proof}
We have $\hat\theta_n\in \{1,\ldots,n\}$ by definition and since $\theta\in(0,1)$, $\hat\theta_n\ne n$ almost surely for $n$ sufficiently large. By Proposition \ref{propest}, $\P(n|\hat\theta_n-\theta|=\KLEINO(n))\to 1$, as $n\to\infty$. Thus, we obtain that
\begin{align*}Q_n&=\frac{\frac{1}{\theta n}\sum_{j=1}^{\theta n}(\Delta_j^n Z)^2}{\frac{1}{(n-\theta n)}\sum_{j=\theta n+1}^{n}(\Delta_j^n Z)^2}\big(1+\KLEINO_{a.s.}(1)\big)\\
&=\frac{\E[(\Delta_1^n Z)^2]}{\E[(\Delta_n^n Z)^2]}\big(1+\KLEINO_{\P}(1)\big)=\begin{cases}n^{-2(H-\tilde H)} \,,&~\mbox{under}~H_1\\ \frac{\sigma^2}{\tilde\sigma^2} \,,&~\mbox{under}~H_1^{\sigma}\end{cases}\,.\hfill\qedhere
\end{align*}
\end{proof}
Analogous results can be proved for statistics based on \eqref{sscusum} along the same lines.

%\clearpage
\section{Simulations and data example\label{sec:6}}
\begin{figure}[t]
\begin{center}
\fbox{
\includegraphics[width=5.64cm]{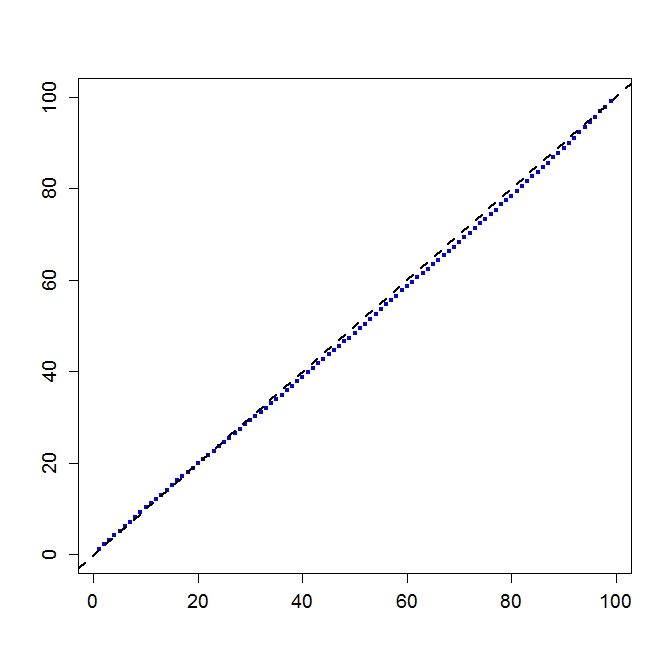}~~~~\includegraphics[width=5.928cm]{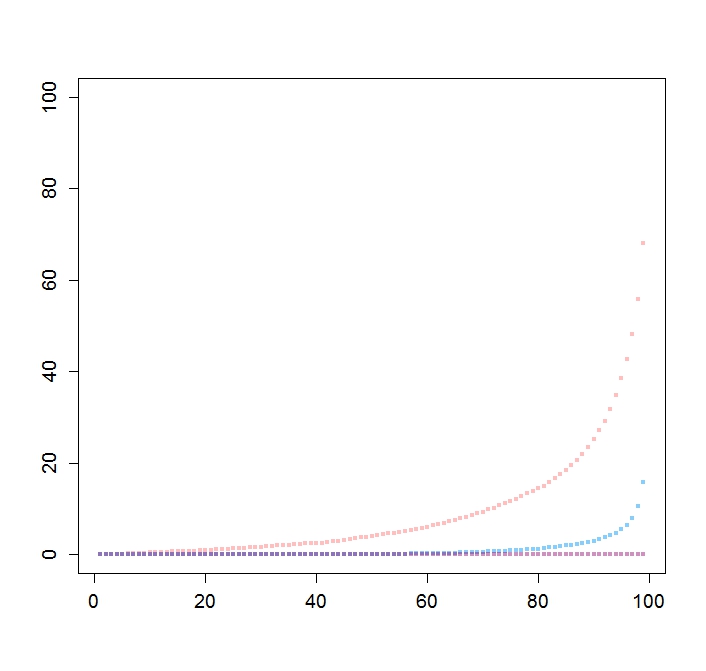}}
\caption{\label{Fig:1}Left: Empirical size of the test under $H_0$. %Middle: Boxplot comparing performances of three estimators of the Hurst exponent.
Right: Empirical type II error rates for the test under $H_1$.}
\end{center}\end{figure} 

%as the test statistic which (under the null that the volatility is constant) tends as $n\rightarrow\infty$ to a Kolmogorov-Smirnov law; see \cite{eval}. 
%Under the alternative $T_n$ diverges almost surely.
We simulate discretizations of $B_{\sigma}^H$ using the code by \cite{fBm2} with the Cholesky method. 
%First, the plot in the middle of Figure \ref{Fig:1} compares the performance of three related estimators for the Hurst exponent:
%\begin{enumerate}
%\item $\hat H_n$ from \eqref{Hest1}.
%\item $\hat H_n^{(\sigma)}$ from \eqref{Hest2}.
%\item The increments ratio estimator from \cite{IR} using second-order increments.
%\end{enumerate}
The parameter configuration under the null hypothesis $H_0$ is $\sigma=2$ and $H=0.2$. %The boxplots based on 10,000 Monte Carlo iterations reveal that $\hat H_n$ has a much smaller finite sample variance compared to the two other estimators for $n=100$. The robustification for general $\sigma^2$ from \eqref{Hest2} has a much larger variance. The more sophisticated increments ratio estimator has a better performance, but still cannot keep up with $\hat H_n$. According to \cite{brouste} the mle, which can be computed for moderate sample sizes, allows to attain a smaller variance for general $\sigma^2$.% which should in this case be close to that of $\hat H_n$.
%Here, we change the parameter configuration to $\sigma=1$ and $H=0.2$, because for $\sigma\ne 1$ estimator $\hat H_n$ from \eqref{Hest1} would be biased. The boxplot reveals that $\hat H_n$ has a much smaller finite sample variance compared to the two other estimators for $n=100$. The robustification for general $\sigma$ from \eqref{Hest2} has a much larger variance. The more sophisticated increments ratio estimator has a better performance, but still cannot keep up with $\hat H_n$. According to \cite{brouste} maximum-likelihood, which can be computed for moderate sample sizes, allows to attain a smaller variance for general $\sigma^2$.% which should in this case be close to that of $\hat H_n$. 
We illustrate the empirical size and power of the test by plotting the relative amount of realizations of $T_n$ from \eqref{scusum} smaller or equal to the percentiles of the Kolmogorov-Smirnov law against those percentiles. The results are given in Figure \ref{Fig:1} under the null hypothesis and alternative hypotheses, respectively.

For $n=100$, under $H_0$, the left plot in Figure \ref{Fig:1} confirms a highly accurate fit by the Kolmogorov-Smirnov limit law based on 10,000 Monte Carlo iterations. For other values of $H\in(0,3/4)$, we obtain as well a high accuracy of the fit. Figure \ref{Fig:2} shows histograms of 10,000 Monte Carlo iterations under $H_0$, for $\sigma=2$, $H=0.2$ and $n=100$, and different alternative hypotheses, when $n=100$ for $\tilde H=0.3$ and $\tilde H=0.4$, and when $n=1,000$ for $\tilde H=0.3$ and $\tilde H=0.4$. As expected, the power increases with increasing $|\tilde H-H|$ and with larger sample sizes. The left histogram under $H_0$ closely tracks the asymptotic Kolmogorov-Smirnov law. The factors between medians (or quantiles close to the center) for the same changes $H-\tilde H$, and for the two different sample sizes, are close to the expected factors $\sqrt{n}$ determined by the rate we have found in Section \ref{sec:3}. For the large sample sizes, the histograms under $H_1$ show fat tails which, however, barely reduce the power of the test in these cases. The right plot in Figure \ref{Fig:1} depicts the percentage type II error rates for the same specifications of $H_1$ as in Figure \ref{Fig:2}. The testing levels of interest are located on the right of the x-axis, that is, the x-axis gives the percentage values of $(1-\alpha)$ when $\alpha$ is the testing level. For sample size $n=1,000$, we basically have a power of one at all reasonable testing levels in both considered cases. The curves for $n=100$ with $\tilde H=0.3$ and $\tilde H=0.4$, are hence more informative. The ordering of the different curves is clear as we have larger power for larger  $|H-\tilde H|$ and larger $n$ (compare to the histograms). Using $T_n^{(2)}$ instead of $T_n$, we obtain analogous empirical distributions under the alternative hypotheses while the size under the null hypothesis is a bit better for $T_n$. However, when $H=0.85$ under the null hypothesis, the left plot of Figure \ref{Fig:3} shows that the empirical distribution of $T_n$ is not close to the Kolmogorov-Smirnov law any more, while the empirical distribution of $T_n^{(2)}$ attains a good size. Under alternative hypotheses with a change of large Hurst exponents, both statistics show similar empirical distributions again. Overall, a very good power for moderate finite sample sizes of our tests is confirmed. This is in line with classical findings in change-point theory for parametric cusum tests.\\[.2cm]
\begin{figure}[t]
\centering\includegraphics[width=10cm]{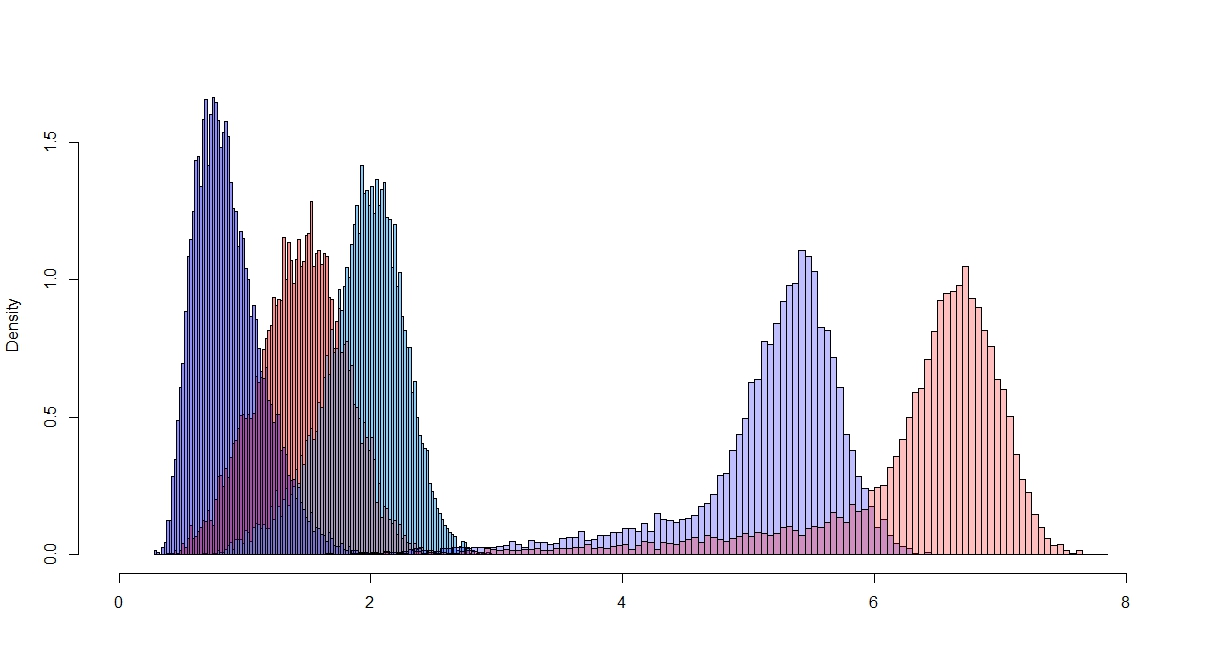}
\caption{\label{Fig:2}Left: Histogram for realizations of $T_n$ from \eqref{scusum} under $H_0$ and $H_1$ for $n=100$ with $\tilde H=0.3$ and $\tilde H=0.4$, and for $n=1,000$ with $\tilde H=0.3$ and $\tilde H=0.4$ (left to right).}\end{figure} 
For a real data example, we use the daily total sunspot number\footnote{Source: WDC-SILSO, Royal Observatory of Belgium, Brussels, \url{http://www.sidc.be/silso/datafiles}, accessed on April 3, 2019.} from 1848/12/23 to 2019/03/31 having $n=62,190$ daily observations. Time series of sunspots are often analyzed using fBm models, see, for instance, \cite{ssnfractals}. Estimator \eqref{Hest2} and the second-order increments ratio estimator from \cite{IR} applied to all data both yield $H\approx 0.469$. However, $T_n^{(2)}\approx 21.6$, such that we clearly reject the hypothesis that $H$ is constant. A classical R/S estimate from the increments yields a different value of ca.\ 0.32 for $H$. This method is often used in the applied literature. We do not rely on this R/S method, however, since \cite{bardet2018} notes that ``a convincing asymptotic study of such an estimator'' does \emph{not} exist, and \cite{taqqufractals} have demonstrated that the obtained estimates are in general not accurate. Sunspots have a periodic behavior with at least one cycle of about 11 years, compare Figure \ref{Fig:3}. Especially for non-stationary time series, we obtained inaccurate R/S Hurst exponent estimates also in simulation experiments. \\
The right-plot of Figure \ref{Fig:3} shows 21 point estimates for the Hurst exponent on time blocks with 3,000 days. Since more than 10,000 daily increments and around 2,000 second-order increments are zero, we need to adjust the increments ratio estimator and we see some relevant differences between both estimators. Nevertheless, the empirical findings allow to conclude that a multifractional Brownian motion model is better suited and that the Hurst exponent is not constant and has larger values in a period after 1950 than before. Looking at sub-samples of the time series, we find that for 3-years block length the test does not reject the hypothesis of a constant Hurst exponent in these blocks at 10\%-level in 7 of the 63 blocks with a minimum value of $T^{(2)}_n\approx 0.65$. For 1-year block length it is not rejected for 81 from 170 years. This indicates that within smaller time intervals fBm may be used as a suitable model.

%there still do not really exist a convincing asymptotic studyof such an estimator, have shown that this estimatoris not really accurate
%\cite{ssnfractals} \cite{taqqufractal}

\begin{figure}[t]
\fbox{
\includegraphics[width=4.55cm]{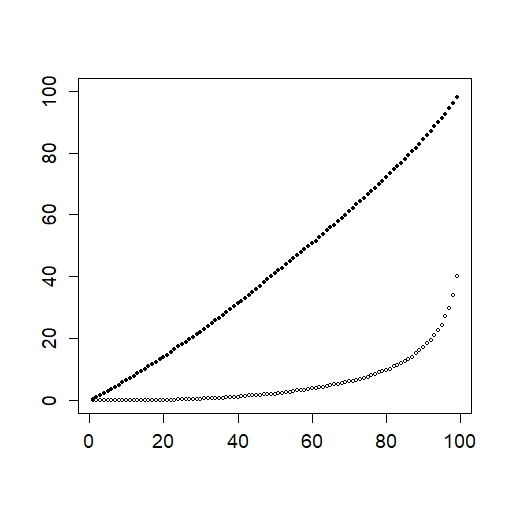}~\includegraphics[width=4.7cm]{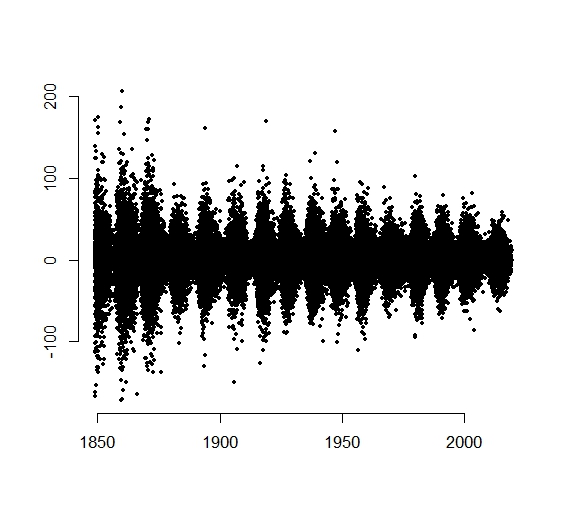}~\includegraphics[width=4.94cm]{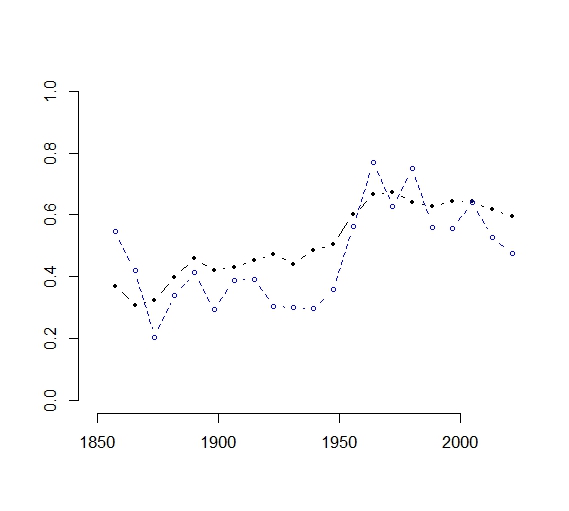}}
\caption{\label{Fig:3}Left: Empirical size of tests based on $T_n^{(2)}$ (black dots) and $T_n$ under $H_0$ with $H=0.85$. Middle: Sunspot daily increments.
Right: Estimated Hurst exponents on blocks by \eqref{Hest2} (back dots) and increment ratios.}\end{figure} 

%\clearpage
%\section*{\refname}

\addcontentsline{toc}{section}{References}
\bibliographystyle{chicago}%{imsart-number}
\bibliography{literatur}
\end{document}